\documentclass[letterpaper,12pt]{article}

\usepackage[english]{babel}
\usepackage[utf8x]{inputenc}
\usepackage[T1]{fontenc}
\usepackage[affil-it]{authblk}

\usepackage[letterpaper,top=3cm,bottom=3cm,left=3cm,right=3cm,marginparwidth=1.75cm]{geometry}

\usepackage{amsmath}
\usepackage{amsthm}
\usepackage{amssymb}
\usepackage[all,cmtip]{xy}
\usepackage{graphicx}
\usepackage[colorinlistoftodos]{todonotes}
\usepackage[colorlinks=true, allcolors=blue]{hyperref}
\usepackage{tikz-cd}
\usepackage{theoremref}

\theoremstyle{plain}
\newtheorem{theorem}{Theorem}[section]
\newtheorem{lemma}[theorem]{Lemma}
\newtheorem{prop}[theorem]{Proposition}

\theoremstyle{definition}
\newtheorem{exx}[theorem]{Example}
\newtheorem{deff}[theorem]{Definition}
\newtheorem{rmk}[theorem]{Remark}

\newcommand{\F}{\mathbb{F}}
\newcommand{\Q}{\mathbb{Q}}

\newcommand{\E}{\mathbb{E}}
\newcommand{\Z}{\mathbb{Z}}

\newcommand{\AsCatA}{\mathcal{P}_A}
\newcommand{\AsCatC}{\mathcal{P}_C}

\newcommand{\RC}{R_{\mathbb{C}}}

\newcommand{\kRF}{kR_{\mathbb{F}}}
\newcommand{\kRQ}{kR_{\mathbb{Q}}}

\newcommand{\GreenCZ}{\mathcal{F}^{\mu}_{\mathcal{C},\Z}}
\newcommand{\GreenDk}{\mathcal{F}^{\mu}_{\mathcal{D},k}}
\newcommand{\StarGDk}{\mathcal{F}^{\mu, \star}_{\mathcal{D},k}}
\newcommand{\invol}{\_^{\star}}
\newcommand{\Dual}{\_^{\bullet}}

\title{Anti-involutions on Green biset functors}

\author{
Robert Boltje\\
Department of Mathematics\\
University of California\\
Santa Cruz, CA 95064\\
U.S.A.\\
\texttt{boltje@ucsc.edu}
\and 
Benjam\'in Garc\'ia\\
Centro de Ciencias Matem\'aticas\\
Universidad Nacional Aut\'onoma de M\'exico\\
Morelia, Michoac\'an 58089\\
Mexico\\
\texttt{benjamingarcia@matmor.unam.mx}}

\begin{document}

\maketitle

\begin{abstract}
In this article, we propose a concept of anti-involution for Green biset functors and we provide equivalent definitions. We present $\star$-Green biset functors and we study their orthogonal units and the orthogonal automorphisms in their associated categories.
\end{abstract}

\section{Introduction}

On many representations rings that we can associate to a finite group, taking \textit{dual representations} induces anti-involutions that are natural with respect to biset operations. In the case of bisets and bimodules, taking \textit{duals} reverses the actions of the groups inducing a contravariant automorphism on the associated category of the respective representation functor. These operations play an important role in the study of unit groups of some representation rings like the Burnside ring and trivial source module rings by means of the notion of \textit{orthogonal unit} \cite{Bar,BolPer,Carman}, and they are a key ingredient in the definition of \textit{$p$-permutation equivalence} \cite{BolPer2}.

We propose an approach to anti-involutions and duality in the context of Green biset functors. The necessary background on Green biset functors is provided in Section 2. Through Section 3, we study associated categories and their linear functors, we present a map of algebras $\widetilde{\_}:A(G)\longrightarrow End_{\AsCatA}(G)$ for a Green $\mathcal{D}$-biset functor $A$ and a group $G\in Ob(\mathcal{D})$, generalizing the map $B(G)\longrightarrow B(G,G)$ given in \cite[2.5.9]{BoucBook}, and we prove that $\AsCatA^{op}\cong \mathcal{P}_{A^{op}}$, where $A^{op}$ is the \textit{opposite Green biset functor}. In Section 4, we define an \textit{anti-involution} on $A$ as an endomorphism $\_^{\star}$ such that each component arrow is an anti-involution of algebras, and we prove that this is the same as having an isomorphism of Green biset functors $\delta:A\longrightarrow A^{op}$ such that $\delta^{op}=\delta^{-1}$, as well as having a linear isomorphism $\_^{\bullet}:\AsCatA\longrightarrow \AsCatA^{op}$ satisfying additional conditions. A \textit{$\star$-Green biset functor} is then defined to be a Green biset functor $A$ together with an anti-involution $\_^{\star}$. Finally, Section 5 is devoted to the study of generalized \textit{orthogonal units} and \textit{orthogonal automorphisms} for a $\star$-Green biset functor $A$, where we deduce several relations from the tilde maps and morphisms of $\star$-Green biset functors.

\section{Green biset functors}

In this section, we provide a brief exposition on Green biset functors. We refer the reader to \cite{BoucRom} or \cite{Gar2} for details. Throughout this text, the letters $G$, $H$, $K$, $L$,... will denote finite groups. Rings and algebras are assumed to be associative and unital, and their homomorphisms send units to units.

For groups $H$ and $G$, $B(H,G)$ denotes the \textit{Burnside group of $(H,G)$-bisets}. Given a $(K,H)$-biset $Y$ and an $(H,G)$-biset $X$, their \textit{composition} is denoted by $Y\times_H X$ or simply $Y\circ X$, and we write $[y,x]$ for the class of a pair $(y,x)$ in $Y\circ X$. Composition of bisets induces associative bi-additive maps
\begin{equation}
\_\circ\_: B(K,H)\times B(H,G)\longrightarrow B(K,G), ([Y],[X])\mapsto [Y\circ X]   
\end{equation}
for all $K$, $H$ and $G$. This leads to the definition of the \textit{biset category} $\mathcal{C}$, having for objects all finite groups and hom-sets $Hom_{\mathcal{C}}(G,H):= B(H,G)$, with composition given by the maps in (1) and $\text{Id}_G=[G]$ (the class of the regular $(G,G)$-biset) for each $G$.

If $X$ is an $(H,G)$-biset, its \textit{opposite} $X^{op}$ is just $X$ as a $(G,H)$-biset for the actions $g\cdot x\cdot h=h^{-1}xg^{-1}$, for $x\in X$, $g\in G$ and $h\in H$. Taking opposites induces isomorphisms of abelian groups
\begin{equation}
\_^{op}:B(H,G)\longrightarrow B(G,H), [X]\mapsto [X^{op}]
\end{equation}
for all $H$ and $G$, satisfying $(\beta \circ \alpha)^{op}=\alpha^{op}\circ \beta^{op}$ and $(\alpha^{op})^{op}=\alpha$ for all $\alpha\in B(H,G)$ and $\beta \in B(K,H)$, and $\text{Id}_G^{op}=\text{Id}_G$. These maps define an isomorphism $\_^{op}:\mathcal{C}\longrightarrow \mathcal{C}^{op}$ which is the identity in objects, satisfying $\_^{op}\circ \_^{op}=Id_{\mathcal{C}}$.

Each $(H,G)$-biset $X$ can be regarded as an $H\times G$-set for the action $(h,g)\cdot x:=hxg^{-1}$, for $x\in X$ and $(h,g)\in H \times G$, denoted by $\overrightarrow{X}$, and each $H\times G$-set $X$ is an $(H,G)$-biset for $h\cdot x\cdot g:=(h,g^{op})x$. So we have isomorphisms of abelian groups
\begin{align}
   \overrightarrow{(\_)}:B(H,G)\longrightarrow B(H\times G), [X]\mapsto [\overrightarrow{X}]
\end{align}
for all $H$ and $G$.

Given an $(H,G)$-biset $X$ and an $(L,K)$-biset $Y$, their \textit{external product} is their cartesian product $X\times Y$ regarded as an $(H\times L,G\times K)$-biset, and again this operation extends to biadditive maps
\begin{equation}
\times:B(H,G)\times B(L,K)\longrightarrow B(H\times L,G\times K), ([X],[Y])\mapsto [X\times Y]
\end{equation}
for all $G$, $H$, $K$ and $L$, inducing a bifunctor $\times:\mathcal{C}\times \mathcal{C} \longrightarrow \mathcal{C}$ for which $\mathcal{C}$ is a symmetric monoidal category.

For a group $G$, the $(G\times G, 1)$-biset $\overrightarrow{G}$ is $G$ with the actions $(g_1,g_2)\cdot g\cdot 1:=g_1gg_2^{-1}$ for $g\in G$, $(g_1,g_2)\in G\times G$, while $\overleftarrow{G}$ stands for its opposite $(1,G\times G)$-biset. If $\phi:G\longrightarrow H$ is a group homomorphism, the \textit{induction by $\phi$}, denoted by $Ind(\phi)$, is $H$ regarded as an $(H,G)$-biset for the group multiplication, and in a similar way, the \textit{restriction by $\phi$} is $H$ as a $(G,H)$-biset, denoted by $Res(\phi)$. If $\phi$ is an isomorphism, we write $Iso(\phi)$ instead of $Ind(\phi)$. These bisets generalize the basic biset operations.

We write $\Delta=\Delta_G:G\longrightarrow G\times G$ for the \textit{diagonal morphism} $g\mapsto (g,g)$, and given groups $G_1$, $G_2$, ..., $G_n$, we write $\rho_{G_1,...,G_n}:\prod_{i=1}^n G_i\longrightarrow \prod_{i=1}^n G_{n + 1-i}$ for the isomorphism that reverses the coordinates. When $n=2$, we will write $\tau_{G,H}$ instead of $\rho_{G,H}$.

\begin{lemma} \thlabel{Rflxns}
The following isomorphisms of bisets hold:
\begin{enumerate}
    \item $Iso(\tau_{G,H})\circ Iso(\tau_{H,G})\cong H\times G$.
    \item $Iso(\rho_{G,H,K,L})\cong Iso(\tau_{H\times G,L\times K})\circ (Iso(\tau_{G,H})\times Iso(\tau_{K,L}))$ 
    
    $\cong  (Iso(\tau_{K,L})\times Iso(\tau_{G,H}))\circ Iso(\tau_{G\times H,K\times L})$.
    \item $Iso(\tau_{G,H})\circ Inf_G^{G\times H}\cong Inf_G^{H\times G}$, $Res^{H\times G}_G\circ Iso(\tau_{G,H})\cong Res^{G\times H}_H$.
    \item $Iso(\tau_{G,G})\circ Ind(\Delta)\cong Ind(\Delta)$, $Res(\Delta)\circ Iso(\tau_{G,G})\cong Res(\Delta)$.   
    \item $Iso(\tau_{G, K})\circ (G\times \overleftarrow{H} \times K)\cong (K\times \overleftarrow{H}\times G)\circ Iso(\rho_{G,H,H,K})$.
    \item $Iso(\tau_{G,G})\circ \overrightarrow{G}\cong \overrightarrow{G}$.
\end{enumerate}
\end{lemma}

\begin{proof}
Assertions (1) - (4) are obvious. To prove (5), notice that the map
$$Iso(\tau_{G, K})\circ (G\times \overleftarrow{H} \times K)\longrightarrow (K\times \overleftarrow{H}\times G)\circ Iso(\rho_{G,H,H,K}),$$
$$[(k_1,g_1),(g_2,h,k_2)]\mapsto [(k_1k_2,h^{-1},g_1g_2),(1,1,1,1)]$$
is an isomorphism of $(K\times G,G\times H\times H\times K)$-bisets with inverse given by
$$[(k_1,h_1,g_1),(k_2,h_2,h_3,g_2)]\mapsto [(1,1),(g_1g_2,h_3^{-1}h_1^{-1}h_2,k_1k_2)],$$
and (6) follows by taking opposites and applying (5).
\end{proof}

\begin{lemma}\thlabel{totheleft} 
Let $Y$ be a $(K,H)$-biset and $X$ be an $(H,G)$-biset.
\begin{enumerate}
    \item $\overrightarrow{Y\circ X}\cong (K\times \overleftarrow{H} \times G)\circ (\overrightarrow{Y}\times \overrightarrow{X})$ as $K\times G$-sets.
    \item $\overrightarrow{\text{Id}_G}\cong \overrightarrow{G} \circ \{\bullet\}$ as $G\times G$-sets.
    \item $\overrightarrow{X^{op}}\cong Iso(\tau_{H,G})\circ \overrightarrow{X}$ as $G\times H$-sets. 
\end{enumerate}
\end{lemma}
\begin{proof}
Parts 1 and 2 are straightforward, and for Part 3, notice that 
$$\overrightarrow{X^{op}}\longrightarrow Iso(\tau_{H,G})\circ \overrightarrow{X},x\mapsto [(1,1),x]$$
is an isomorphism of $G\times H$-sets with inverse $[(g,h),x]\mapsto hxg^{-1}$.
\end{proof}

Given a commutative ring $k$ and a pre-additive subcategory $\mathcal{D}\subseteq \mathcal{C}$, a \textit{$\mathcal{D}$-biset functor} (\textit{over $k$}) is a $k$-linear functor $F:k\mathcal{D}\longrightarrow {_kMod}$, where $k\mathcal{D}$ stands for the $k$-linearization of $\mathcal{D}$. We write $\mathcal{F}_{\mathcal{D},k}$ for the category of $\mathcal{D}$-biset functors over $k$ and natural transformations.

\begin{deff}[Romero] \thlabel{GBF1}
Let $k$ be a commutative ring and $\mathcal{D}\subseteq \mathcal{C}$ be a full subcategory that is closed under taking subquotients and products in objects. A \textit{Green $\mathcal{D}$-biset functor} (\textit{over $k$}) is a $\mathcal{D}$-biset functor $A$ such that $A(G)$ is a unital $k$-algebra for every $G\in Ob(\mathcal{D})$, and for every group homomorphism $\phi:G\longrightarrow H$ with $G,H\in Ob(\mathcal{D})$, the following conditions are satisfied:
    \begin{enumerate}
    \item The map $A(Res(\phi)):A(H)\longrightarrow A(G)$ is a $k$-algebra homomorphism,
    \item (\textit{Frobenius identities}) For all $a\in A(H)$ and all $b\in A(G)$, we have
    $$aA(Ind(\phi))(b)=A(Ind(\phi))(A(Res(\phi))(a)b),$$
    $$A(Ind(\phi))(b)a=A(Ind(\phi))(bA(Res(\phi))(a)).$$
    \end{enumerate}

A \textit{morphism of Green biset functors} is a natural transformation whose component arrows are $k$-algebra homomorphisms, and $\GreenDk$ stands for the \textit{category of Green $\mathcal{D}$-biset functors over $k$}.
\end{deff}

A Green $\mathcal{D}$-biset functor can also be defined as a $\mathcal{D}$-biset functor $A$ endowed with bilinear applications $\times:A(G)\times A(H)\longrightarrow A(G\times H)$, $G,H\in Ob(\mathcal{D})$, which are associative and bifunctorial, with an identity element $\epsilon_A \in A(1)$ (see \cite[Lemma 3]{BoucRom}). To define bilinear products from the given definition, set
\begin{equation}
    a\times b:=A(Inf_G^{G\times H})(a)A(Inf_H^{G\times H})(b)\in A(G\times H)
\end{equation}
for $a\in A(G)$ and $b\in A(H)$, $G,H\in Ob(\mathcal{D})$, taking $\epsilon_A:=1\in A(1)$. Departing from bilinear products, $A(G)$ is a unital associative algebra for the product
\begin{equation}
    ab:=A(Res(\Delta))(a\times b)\in A(G)
\end{equation}
for $a,b\in A(G)$, and $1:=A(Inf_1^G)(\epsilon_A) \in A(G)$.

Many examples of Green biset functors come from represention rings. It is a well-known fact that the ($k$-linearized) Burnside functor $B_k$ is an initial object in $\GreenDk$, and we write $e:B_k\longrightarrow A$ for the only morphism of Green biset functors from $B_k$ to $A$.

\begin{exx}
The \textit{commutant} $CA$ of a Green biset functor $A$ \cite[Definition 18]{BoucRom} is the subfunctor of $A$ given by 
$$CA(G)=\left\{a\in A(G)|a\times b=A(Iso(\tau_{H,G}))(b\times a),\forall b\in A(H),\forall H\in Ob(\mathcal{D})\right\}$$
for $G\in Ob(\mathcal{D})$. It is a Green biset subfunctor whose evaluations are contained in $Z(A(G))$ for every $G$, and $A$ is said to be \textit{commutative} if $CA=A$.
\end{exx}

\begin{exx}
The \textit{opposite} $A^{op}$ of a Green biset functor $A$ is just $A$ as a functor with the opposite multiplication for $A(G)$. It is a Green biset functor, and for every morphism $f:A\longrightarrow C\in \GreenDk$ we have a morphism $f^{op}:A^{op}\longrightarrow C^{op}$ with the same component arrows. The functor $\_^{op}$ is a self-inverse automorphism on $\GreenDk$, and $A^{op}= A$ if and only if $A$ is commutative.
\end{exx}

\begin{exx}
Given a Green biset functor $A$ and positive integers $n$ and $m$, the \textit{functor of $n\times m$-matrices over $A$} is the functor $M_{n\times m}(A)$ sending an object $G$ to $M_{n\times m}(A(G))$ (the set of $n\times m$-matrices with entries in $A(G)$), and sending a morphism $x\in B_k(H,G)$ to the $k$-linear transformation $[a_{i,j}]\mapsto[A(x)(a_{i,j})]$, $[a_{i,j}]\in M_{n\times m}(A(G))$. When $m=n$, $M_{n\times n}(A)$ is a Green biset functor for matrix multiplication, and if $n>1$, then $M_{n\times n}(A)$ is not commutative. 
\end{exx}

\section{Associated categories and linear functors}

The associated category of a Green biset functor $A$ was introduced by Bouc in \cite[Chapter 8]{BoucBook} as a generalization of the biset category. It plays an important role in the study of modules over $A$.

\begin{deff}[Bouc] \thlabel{AssoCat}
For a Green $\mathcal{D}$-biset functor over $k$, $A$, its \textit{associated category} $\mathcal{P}_A$ has the same objects as $\mathcal{D}$ and hom-sets $Hom_{\mathcal{P}_A}(G, H):=A(H\times G)$ for $G,H\in Ob(\mathcal{D})$, with bilinear composition defined by $\alpha\circ \beta:=A(K\times \overleftarrow{H}\times G)(\alpha\times \beta)$, for $\alpha \in A(K\times H)$ and $\beta \in A(H\times G)$, and identity $Id_G=A(\overrightarrow{G})(\epsilon_A)$ for $G$. 
\end{deff}

By parts 1 and 2 of \thref{totheleft}, the maps in (3) define a linear isomorphism $\overrightarrow{(\_)}:k\mathcal{D}\longrightarrow \mathcal{P}_{B_k}$ given as the identity in objects. We now give a generalization of the map $\widetilde{\_}:B(G)\longrightarrow B(G,G)$ presented in \cite[2.5.9]{BoucBook} to an arbitrary Green biset functor.

\begin{prop} \thlabel{map2end}
Let $A$ be a Green biset functor and $G\in Ob(\mathcal{D})$. The map 
$$\widetilde{\_}:A(G)\longrightarrow End_{\AsCatA}(G), a\mapsto \widetilde{a}:= A(Ind(\Delta))(a)$$
is an injective $k$-algebra homomorphism.
\end{prop}
\begin{proof}
Note that $\widetilde{ab}=A(Ind(\Delta)\circ Res(\Delta))(a\times b)$, and
$$\widetilde{a}\circ \widetilde{b}=A((G\times \overleftarrow{G}\times G)\circ Ind(\Delta\times \Delta))(a\times b),$$
for all $a,b\in A(G)$. The right-hand sides of both expressions are equal since
$$Ind(\Delta)\circ Res(\Delta) \longrightarrow (G\times \overleftarrow{G}\times G) \circ Ind(\Delta\times \Delta)$$
$$[(g_1,g_2),(g_3,g_4)]\mapsto [(1,1,1),(g_1g_3,g_3,g_4,g_2g_4)]$$
is an isomorphism of $(G\times G,G\times G)$-bisets, and 
$$\widetilde{1}=A(Ind(\Delta)\circ Inf_1^{G})(\epsilon_A)=A(\overrightarrow{G})(\epsilon_A)=Id_G.$$ 
Since the map $\widetilde{\_}$ is $k$-linear, then it is a $k$-algebra homomorphism. Let $p_1:G\times G \longrightarrow G$ be the projection to the first factor, then the map $A(Ind(p_1))$ is a retraction of $A(Ind(\Delta))$ since $p_1\circ \Delta=Id_G$. 
\end{proof}

\begin{lemma}
Let $A$ be a Green biset functor. Then for all $a,b\in A(G)$, the following holds:
\begin{enumerate}
    \item $A(\widetilde{a})(b)=ab$,
    \item $\widetilde{a}\circ (b\times \epsilon_A)=ab\times \epsilon_A$.
\end{enumerate}
\end{lemma}
\begin{proof}
To prove part (1), observe that
$$A(\widetilde{a})(b)=A((G\times \overleftarrow{G})\circ (Ind(\Delta)\times G))(a\times b).$$

Using the Mackey formula for bisets, it follows easily that $(G\times \overleftarrow{G})\circ (Ind(\Delta)\times G)$ is transitive. The stabilizer of $[(1,1),((1,1),1)]$ is the graphic of $\Delta$, therefore
$$(G\times \overleftarrow{G})\circ (Ind(\Delta)\times G)\cong Res(\Delta),$$
implying that $A(\widetilde{a})(b)=ab$. Part (2) follows by a similar argument.
\end{proof}

For $H,G\in Ob(\mathcal{D})$ and $x\in B_k(H,G)$, we set
\begin{align}
    \E_A(x)(\alpha)=e_{H\times G}(\overrightarrow{x})\circ \alpha\circ e_{G\times H}(\overrightarrow{x^{op}})
\end{align}
for $\alpha\in End_{\AsCatA}(G)$. The map $\E(x):End_{\AsCatA}(G) \longrightarrow End_{\AsCatA}(H)$ is clearly $k$-linear, and it is straightforward that $\E_A(\text{Id}_G)=Id_{End_{\AsCatA}(G)}$ and $\E_A(y\circ x)=\E_A(y)\circ \E_A(x)$ for composable morphisms $y$ and $x$ in $k\mathcal{D}$. Setting $\E_A(G)=End_{\AsCatA}(G)$ for $G\in Ob(\mathcal{D})$, we have a functor $\E_A:k\mathcal{D}\longrightarrow {_k}Mod$ that we call the \textit{functor of double algebras of $A$}. Unfortunately, $\E_A$ is not a biset functor, as it is not linear on morphisms.

\begin{lemma}
Let $A$ be a Green biset functor. Then for every group homomorphism $\phi:G\longrightarrow H$ with $G,H\in Ob(\mathcal{D})$ and $a\in A(G)$, we have
$$\widetilde{A(Ind(\phi))(a)}=\E_A(Ind(\phi))(\widetilde{a})).$$
\end{lemma}
\begin{proof}
Observe that
\begin{align*}
    \E_A(Ind(\phi))(\widetilde{a}) &=e_{H\times G}\left(\overrightarrow{Ind(\phi)}\right)\circ \widetilde{a}\circ e_{G\times H}\left(\overrightarrow{Ind(\phi)^{op}}\right)\\
    &=e_{H\times G}\left(\overrightarrow{Ind(\phi)}\right)\circ \widetilde{a}\circ e_{G\times H}\left(\overrightarrow{Res(\phi)}\right)\\
    &=A\left(\left(H\times \overleftarrow{G}\times \overleftarrow{G}\times H\right)\circ \left(\overrightarrow{Ind(\phi)}\times Ind(\Delta_G)\times \overrightarrow{Res(\phi)} \right)\right)(a),
\end{align*}
while $\widetilde{A(Ind(\phi))(a)}=A\left(Ind(\Delta_H)\circ Ind(\phi)\right)(a)$. Then the map
$$Ind(\Delta_H)\circ Ind(\phi)\longrightarrow \left(H\times \overleftarrow{G}\times \overleftarrow{G}\times H\right)\circ \left(\overrightarrow{Ind(\phi)}\times Ind(\Delta_G)\times \overrightarrow{Res(\phi)} \right),$$
$$[(h_1,h_2),h_3]\mapsto [(h_1,1,1,h_2), (h_3,(1,1),h_3^{-1})],$$
is morphism of $(H,G)$-bisets. It is surjective because $Ind(\phi)$ is left-transitive, and it is injective because $Ind(\phi)$ is left-free, thus it is an isomorphism.
\end{proof}

Every morphism of Green biset functors $f:A\longrightarrow C$ induces a linear functor $\mathcal{P}_f:\mathcal{P}_A\longrightarrow \mathcal{P}_C$ given by $\mathcal{P}_f(G)=G$ for $G\in Ob(\mathcal{D})$ and $\mathcal{P}_f(\alpha)=f_{H\times G}(\alpha)$ for $\alpha \in \AsCatA(G,H)$, and if $g:C\longrightarrow D$ is another morphism in $\GreenDk$, we have that $\mathcal{P}_{g\circ f}=\mathcal{P}_g\circ \mathcal{P}_f$ and $\mathcal{P}_{Id_A}=Id_{\AsCatA}$ (see \cite[Section 3]{Gar2}). Given a linear functor $F:\mathcal{P}_A\longrightarrow \mathcal{P}_C$, we may now wonder whether there is a morphism $f:A\longrightarrow C$ such that $F=\mathcal{P}_f$.

\begin{lemma}\thlabel{LiftFun}
Let $A$ and $C$ be Green biset functors and let $F: \mathcal{P}_A\longrightarrow \mathcal{P}_C$ be a $k$-linear functor. Then $F$ is induced by a morphism of Green biset functors if and only if $F(G)=G$ for all $G$, and $F(A(x)(\alpha))=C(x)( F(\alpha))$ for all $x\in B_k(L\times K,H\times G)$ and all $\alpha \in A(H\times G)$.
\end{lemma}
\begin{proof}
The implication is clear. For the converse, recall from \cite[Prop. 4.2]{Gar2} that $A\left(Inf_{\_}^{\_\times 1}\right):A\longrightarrow A_1$ and $C\left(Res^{\_\times 1}_{\_}\right):C_1\longrightarrow C$ are morphisms of Green biset functors, where $A_1$ and $C_1$ are the shifted functors at the trivial group. The maps $f_{0,G}:A(G\times 1)\longrightarrow C(G\times 1), \alpha \mapsto F(\alpha)$ for $G\in Ob(\mathcal{D})$ clearly define a morphism of biset functors $f_0:A_1\longrightarrow C_1$. If $\alpha,\beta\in A_1(G)$, then $\alpha=a\times \epsilon$ and $\beta=b\times \epsilon$ for unique elements $a,b\in A(G)$, and $\alpha\beta=(ab)\times \epsilon=\widetilde{a}\circ \beta$. Then $
f_{0,G}(\alpha\beta) = F(\widetilde{a}\circ \beta)=F(\widetilde{a})\circ F(\beta)=F(\alpha)F(\beta)=f_{0,G}(\alpha)f_{0,G}(\beta)$, and noting that $\overrightarrow{G\times 1}\cong \overrightarrow{G\times 1}\circ \overleftarrow{ 1}\circ \overrightarrow{1}$, it is immediate that $f_{0,G}(1)=1$, hence $f_0$ is a morphism of Green biset functors. Taking $f=C\left(Res^{\_\times 1}_{\_}\right)\circ f_0\circ A\left(Inf_{\_}^{\_\times 1}\right):A\longrightarrow C$, we have that $P_f=F$.
\end{proof}

The opposite category of $\AsCatA$ can be realized as the associated category of $A^{op}$. 

\begin{prop}\thlabel{OppAssCat}
$\mathcal{P}_{A^{op}}\cong \mathcal{P}_A^{op}$ as linear categories.
\end{prop}
\begin{proof}
Writing  $*$ for the multiplication in $A(G)^{op}$ and $\mathring{\times}$ for the bilinear product of $A^{op}$, we have
\begin{align*}
    a \mathring{\times} b &= A^{op}(Inf_G^{G\times H})(a) * A^{op}(Inf^{G\times H}_H)(b)\\
    &=A(Inf^{G\times H}_H)(b) A(Inf_G^{G\times H})(a)\\
    &=A(Iso(\tau_{H,G}))(A(Inf^{H\times G}_H)(b) A(Inf_G^{H\times G})(a))\\
    &=A(Iso(\tau_{H,G}))(b\times a)
\end{align*}
for all $a\in A^{op}(G)$ and $b\in A^{op}(H)$. If $\diamond$  denotes the composition in $\mathcal{P}_{A^{op}}$, then
\begin{align*}
    \alpha\diamond \beta &=A^{op}(K\times \overleftarrow{H}\times G)(\alpha\mathring{\times}\beta)\\
    &=A((K\times \overleftarrow{H}\times G)\circ Iso(\tau_{H\times G,K\times H}))(\beta\times \alpha)\\
    &=A(Iso(\tau_{G,K})\circ (G\times \overleftarrow{H}\times K))(A(Iso(\tau_{H,G}))(\beta)\times A(Iso(\tau_{K,H}))(\alpha))\\
    &=A(Iso(\tau_{G,K}))(A(Iso(\tau_{H,G}))(\beta)\circ A(Iso(\tau_{K,H}))(\alpha))
\end{align*}
for $\alpha \in A^{op}(K\times H)$ and $\beta \in A^{op}(H\times G)$, and thus we have that
$$A(Iso(\tau_{K,G}))(\alpha\diamond \beta)=A(Iso(\tau_{H,G}))(\beta)\circ A(Iso(\tau_{K,H}))(\alpha).$$
Setting $T(G):=G$ for $G\in Ob(\mathcal{D})$, the maps 
$$T_{G,H}:=A(Iso(\tau_{H,G})):\mathcal{P}_{A^{op}}(G, H)\longrightarrow \mathcal{P}_A^{op}(G,H)$$ 
define an isomorphism of linear categories $T:\mathcal{P}_{A^{op}}\longrightarrow \mathcal{P}_A^{op}$.
\end{proof}

\section{Anti-involutions}

\begin{deff}\thlabel{invol}
An \textit{anti-involution} on a Green $\mathcal{D}$-biset functor $A$ is an endomorphism $\_^{\star}:A\longrightarrow A$ as biset functor such that $a^{\star\star}=a$ and $(ab)^{\star}= b^{\star}a^{\star}$ for all $a,b\in A(G)$ and all $G\in Ob(\mathcal{D})$. 
\end{deff}

It is clear that $\invol$ is an automorphism, and it can be easily verified that $1^{\star}=1\in A(G)$ for all $G$, $a^{\star}\times b=(a\times b^{\star})^{\star}$ and $a\times b^{\star}=(a^{\star}\times b)^{\star}$ for all $a\in A(G)$ and $b\in A(H)$. 

\begin{exx}
If $\F$ is a field, taking duals of finitely generated $\F G$-modules commutes with direct sums, tensor products over $\F$ and admissible biset operations, therefore it induces anti-involutions $\_^*$ on the functor of linear $\F$-representations $\kRF$ if $char\;\F=0$ and the functor of trivial source modules $kpp_{\F}$ if $char\;\F=p>0$.
\end{exx}

\begin{exx}
If $A$ is a Green biset functor with an anti-involution $\_^{\star}$, then there is an anti-involution on $M_{n\times n}(A)$ given by $[a_{i,j}]^{\star}=[a_{i,j}^{\star}]^{T}$ (the $T$ stands for the transpose of a matrix) for $[a_{i,j}]\in M_{n\times n}(A)(G)$ and $G\in Ob(\mathcal{D})$. If $n>1$, then $[a_{i,j}]\mapsto [a_{i,j}]^T$ and $[a_{i,j}]\mapsto [a_{i,j}^{\star}]$ are not anti-involutions.
\end{exx}

\begin{prop}\thlabel{AntInvProps}
Let $A$ be a Green biset functor with an anti-involution $\invol$. Then $\invol$ can be restricted to an anti-involution on $CA$, and $\invol$ is an automorphism of $A$ as a Green biset functor if and only if $A$ is commutative. Moreover:
\begin{enumerate}
    \item The identity of $A$ is an anti-involution if and only if $A$ is commutative;
    \item The identity morphism is the only anti-involution on $B_k$;
    \item If $e:B_k\longrightarrow A$ is the initial morphism, then $\invol\circ e=e$.
\end{enumerate}
\end{prop}

\begin{proof}
If $a\in CA(G)$, then for all $b\in A(H)$, $H\in Ob(\mathcal{D})$, we have
$$a^{\star}\times b=(a\times b^{\star})^{\star}=\left(Iso(\tau_{H,G})(b^{\star}\times a)\right)^{\star}=Iso(\tau_{H,G})(b\times a^{\star}),$$ 
hence $a^{\star}\in CA(G)$, and so we can restrict $\invol$ to $CA$. If $A$ is commutative, then $(ab)^{\star}=a^{\star}b^{\star}$ for all $a,b\in A(G)$. Conversely, if $\invol$ is a morphism of Green biset functors, then $(ab)^{\star}=a^{\star}b^{\star}=(ba)^{\star}$ for all $a,b\in A(G)$, thus $ab=ba$ and so $A$ is commutative, proving 1, and now 2 follows easily. Finally, the image of the composite $\invol \circ e:B_k\longrightarrow A$ lies in $CA$, hence it is a morphism of Green biset functors, thus equal to $e$, proving 3.
\end{proof}

\begin{exx}
Let $\mathcal{D}=\mathcal{C}$ and $k=\Z$. Since $\RC$ is commutative, any anti-involution on it is either an automorphism of order $2$ or the identity. We already know the one given by taking duals, but there are uncountable many more. There is a monoid isomorphism from the Prüfer ring $\widehat{\Z}$ onto the endomorphisms of $\RC:grp^{op}\longrightarrow Ring$, sending a coherent sequence $\sigma = ([\sigma_n])_{n}$ to the natural transformation $\Psi^{\sigma}$ whose $G$-th component arrow is the Adams operation $\Psi^{\sigma_{exp(G)}}$ \cite{Bol1}. It is straightforward that the automorphisms of $\RC$ in $\GreenCZ$ are precisely those of $\RC:grp^{op}\longrightarrow Ring$, and that the units of order $2$ in $\widehat{\Z}$ are in bijection with non-empty sets of primes.
\end{exx}

\begin{theorem}
For a Green biset functor $A$, the following statements are equivalent:
\begin{enumerate}
    \item There is an anti-involution $\invol$ on $A$;
    \item There is an isomorphism of Green biset functors $\delta:A\longrightarrow A^{op}$ such that $\delta^{op}= \delta^{-1}$;
    \item There is a $k$-linear functor $\Dual:\AsCatA \longrightarrow \AsCatA^{op}$ which is the identity in objects, satisfying the following:
    \begin{enumerate}
    \item $(\Dual)^{op}\circ \Dual =Id_{\mathcal{P}_A}$,
    \item The diagram
    \begin{align}
        \xymatrix{
        \AsCatA(G,H) \ar[rr]^{T^{-1}\circ \Dual}\ar[d]_{A(x)} &&\mathcal{P}_{A^{op}}(G,H)\ar[d]^{A^{op}(x)}\\
     \AsCatA(K,L)\ar[rr]_{T^{-1}\circ \Dual} &&\mathcal{P}_{A^{op}}(K,L)
    }
    \end{align}
    commutes for all $x\in B_k(L\times K,H\times G)$ and all $G,H,K,L\in Ob(\mathcal{D})$.
\end{enumerate}
\end{enumerate}
\end{theorem}
\begin{proof}
Statements (1) and (2) are clearly equivalent if we take $\delta =\invol$. We now prove that (2) and (3) are equivalent. Taking $\Dual=T\circ \mathcal{P}_{\delta}$, we observe that the diagram
$$
\xymatrix{
A(H\times G) \ar[rr]^-{\delta_{H\times G}} \ar[d]_{A(Iso(\tau_{H,G}))}&&A^{op}(H\times G)\ar[d]^{A(Iso(\tau_{H,G}))} \\
A^{op}(G\times H)\ar[rr]_-{\delta^{op}_{G\times H}} &&A(G\times H) }
$$
commutes by naturality of $\delta$, this translates into a commutative diagram 
$$
\xymatrix{
\AsCatA \ar[rr]^-{\mathcal{P}_{\delta}} \ar[d]_{(T^{-1})^{op}}\ar@{.>}[rrd]^{\Dual} &&\mathcal{P}_{A^{op}}\ar[d]^T \\
\mathcal{P}_{A^{op}}^{op}\ar[rr]_-{\mathcal{P}_{\delta^{op}}^{op}} &&\AsCatA^{op} }
$$
of linear categories and linear functors, and so $\Dual= T\circ \mathcal{P}_{\delta}=\mathcal{P}_{\delta^{op}}^{op}\circ (T^{-1})^{op}$. Hence, 
\begin{align*}
    (\Dual)^{op}\circ \Dual &=(\mathcal{P}_{\delta^{op}}^{op}\circ (T^{-1})^{op})^{op}\circ (T\circ \mathcal{P}_{\delta})\\
    &=\mathcal{P}_{\delta^{op}}\circ T^{-1}\circ T\circ \mathcal{P}_{\delta}\\
    &=\mathcal{P}_{\delta^{op}\circ \delta}\\
    &=Id_{\mathcal{P}_A},
\end{align*}
and since $T^{-1}\circ \Dual=\mathcal{P}_{\delta}$ commutes with bisets, this proves (3). Conversely, by property (b) and \thref{LiftFun}, there is a morphism of Green biset functors $\delta:A\longrightarrow A^{op}$ such that $T^{-1}\circ \Dual=\mathcal{P}_{\delta}$. Taking $x=Iso(\tau_{H,G})$ in Diagram (8), we see that $T^{-1}\circ \Dual=(\Dual)^{op}\circ T$, and since $\delta_{H\times G}^{op}=\delta_{H\times G}$ as maps, then $\delta^{op}\circ \delta=Id_A$, proving (2).
\end{proof}

In analogy to $\star$-algebras, we present the following definition.

\begin{deff}
A \textit{$\star$-Green biset functor} is a Green $\mathcal{D}$-biset functor $A$ with an anti-involution $\_^{\star}$. If $C$ is another $\star$-Green biset functor, a \textit{morphism of $\star$-Green biset functors} is a morphism $f:A\longrightarrow C \in \GreenDk$ such that $f_G(a^{\star})=f_G(a)^{\star}$ for all $a\in A(G)$, $G\in Ob(\mathcal{D})$. The \textit{category of $\star$-Green biset functors} is denoted by $\StarGDk$.
\end{deff}

From now on, when we say that \textit{$A$ is a $\star$-Green biset functor}, it is implicitly stated that there is a chosen anti-involution $\invol$ and $\Dual=T\circ \mathcal{P}_{\delta}$ is its corresponding duality.

\begin{lemma}
Let $A$ be a $\star$-Green biset functor. Then $\Dual\circ \mathcal{P}_e=T\circ \mathcal{P}_e$.
\end{lemma}
\begin{proof}
$\Dual\circ \mathcal{P}_e=(T\circ \mathcal{P}_{\delta})\circ \mathcal{P}_e=T\circ \mathcal{P}_{\delta\circ e}=T\circ \mathcal{P}_{e}$ by part (3) of \thref{AntInvProps}.
\end{proof}

\begin{rmk}
If $\F$ is a field and $M$ is a finitely generated $(\F H,\F G)$-bimodule, $M^{\circ}$ is its dual bimodule and $\overrightarrow{M}$ denotes $M$ regarded as an $\F (H\times G)$-module, then it is straightforward to verify that $\overrightarrow{\left(M^{\circ}\right)}\cong \F Iso(\tau_{H,G})\bigotimes_{\F(H\times G)}\left(\overrightarrow{M}\right)^*$ as $\F (G\times H)$-modules. This is why we consider the functor $\Dual =T\circ \mathcal{P}_{\delta}$ instead of $\mathcal{P}_{\delta}$ alone. 
\end{rmk}

\begin{exx}
If $A$ is a $\star$-Green biset functor, then for every $G$ in $Ob(\mathcal{D})$, the essential algebra $\widehat{A}(G)$ (see \cite{Gar2}) is a $\star$-algebra for the anti-involution $\widehat{\alpha}^{\star}:=\widehat{\alpha^{\bullet}}$, for $\alpha\in A(G\times G)$, which is well-defined since $(I_A(G))^{\bullet}\subseteq I_A(G)$. 
\end{exx}

\begin{lemma}\thlabel{Dmorphd}
If $A$ is a $\star$-Green biset functor, then $\widetilde{(a^{\star})}=(\widetilde{a})^{\bullet}$ for all $a\in A(G)$ and $G\in Ob(\mathcal{D})$.
\end{lemma}
\begin{proof}
Note that the diagram
$$
\xymatrix{
A(G) \ar[rr]^-{A(Ind(\Delta))} \ar[d]_{\invol} &&A(G\times G)\ar[rr]^{A(Iso(\tau_{G,G}))}\ar[d]_{\invol}\ar[rrd]^{\Dual} &&A(G\times G) \ar[d]^{\invol}\\
A(G) \ar[rr]_-{A(Ind(\Delta))}\ar@/_2pc/[rrrr]_{A(Ind(\Delta))} &&A(G\times G)\ar[rr]_{A(Iso(\tau_{G,G}))} &&A(G\times G)\\
&&&&
}
$$
commutes by the naturality of $\invol$ and part (6) of \thref{Rflxns}. 
\end{proof}

For a $\star$-Green biset functor $A$, we set
$$\mathfrak{Re}\;A(G)=\{a\in A(G)|a^{\star}=a\},$$
$$\mathfrak{Im}\;A(G)=\{a\in A(G)|a^{\star}=-a\}$$
for each $G\in Ob(\mathcal{D})$. This defines subfunctors $\mathfrak{Re}\;A$ and $\mathfrak{Im}\;A$ of $A$, since these are the kernels of the endomorphisms $Id_A-\invol$ and $Id_A+\invol$, respectively. Moreover, we have $Im(Id_A+\invol)\subseteq \mathfrak{Re}\;A$ and $Im(Id_A-\invol)\subseteq \mathfrak{Im}\;A$. 

\begin{lemma}\thlabel{PropReIm}
If $A$ is a $\star$-Green biset functor, then $e(B_k)\subset \mathfrak{Re}\;A$. In particular, $\mathfrak{Re}\;A\neq 0$ if $A\neq 0$.
\end{lemma}
\begin{proof}
From part (3) of \thref{AntInvProps}, $e_G(x)^{\star}=e_G(x)$ for all $x\in B(G)$ and all $G$.
\end{proof}

\begin{prop}\thlabel{Re+Im}
Let $A$ be a $\star$-Green biset functor. If $2$ is invertible in $k$, then $A\cong \mathfrak{Re}\;A\oplus \mathfrak{Im}\;A$.
\end{prop}
\begin{proof}
The morphisms $\rho=2^{-1}(Id_A+\invol)$ and $\iota=2^{-1}(Id_A-\invol)$ are orthogonal idempotent endomorphisms of $A$ such that $\rho+\iota=Id_A$; moreover, $Im\;\rho=\mathfrak{Re}\;A$ and $Im\;\iota=\mathfrak{Im}\;A$.
\end{proof}

\begin{exx}
If $k$ is a field of characteristic zero, then $\kRQ$ is a simple biset functor, and by \thref{PropReIm} and \thref{Re+Im}, we have that $\mathfrak{Re}\;\kRQ=\kRQ$ for any anti-involution $\invol$, proving that $\invol=Id_{\kRQ}$. Now, if $\invol$ is an anti-involution on $R_{\Q}$, then $Id_{k}\otimes \invol$ is an anti-involution on $\kRQ$ and so it has to be the identity, therefore $\invol=Id_{R_{\Q}}$. 
\end{exx}


\section{Generalized orthogonal units}

For a $\star$-Green biset functor  $A$, we have induced involutions on $A(G)^{\times}$ and on $Aut_{\AsCatA}(G)$ by the corresponding duality. This motivates the following definition.

\begin{deff}
Let $A$ be a $\star$-Green biset functor.
\begin{enumerate}
    \item An element $u\in A(G)^{\times}$ is said to be \textit{orthogonal} (\textit{with respect to $\invol$}) if $u^{\star}=u^{-1}$. We write $A(G)^{\times,\perp}$ for the set of orthogonal units of $A(G)$.
    \item An isomorphism $\omega \in \AsCatA$ is said to be \textit{orthogonal} (\textit{with respect to $\Dual$}) if $\omega^{\bullet}=\omega^{-1}$. We write $Iso_{\AsCatA}(G,H)^{\perp}$ for the set of orthogonal isomorphisms from $G$ to $H$, and $Aut_{\AsCatA}(G)^{\perp}$ for the set of orthogonal automorphisms of $G$ in $\AsCatA$.
\end{enumerate}
\end{deff}


\begin{prop}
    Let $A$ be a $\star$-Green biset functor and $G\in Ob(\mathcal{D})$.
    \begin{enumerate}
        \item The set $A(G)^{\times,\perp}$ is a subgroup of $A(G)^{\times}$,
        \item Composition and inverses of orthogonal isomorphisms are orthogonal. In particular, the set $Aut_{\AsCatA}(G)^{\perp}$ is a subgroup of $Aut_{\AsCatA}(G)$.
    \end{enumerate}
\end{prop}
\begin{proof}
If $u,v\in A(G)^{\times,\perp}$, then $(uv)^{\star}=v^{\star}u^{\star}=v^{-1}u^{-1}=(uv)^{-1}$, and $(u^{-1})^{\star}=(u^{\star})^{\star}=u$, which proves (1). If $\beta \in  Iso_{\AsCatA}(H,K)^{\perp}$ and $\alpha \in  Iso_{\AsCatA}(G,H)^{\perp}$, then $(\beta\circ\alpha)^{\bullet}=\alpha^{\bullet}\circ \beta^{\bullet}=\alpha^{-1}\circ \beta^{-1}=(\beta\circ \alpha)^{-1}$ and $(\alpha^{-1})^{\bullet}=\alpha^{\bullet\bullet}=\alpha$, proving (2). 
\end{proof}

\begin{lemma} \thlabel{involnmorph}
Let $A$ be a $\star$-Green biset functor. If $\phi:G\longrightarrow H$ is a group homomorphism with $G,H\in Ob(\mathcal{D})$, then $A(Res(\phi))$ can be restricted to a group homomorphism $A(H)^{\times,\perp} \longrightarrow A(G)^{\times,\perp}$. Moreover, if $\phi$ is an epimorphism, then $A(Res(\phi))$ is a monomorphism, and if $u\in A(H)^{\times}$ is such that $A(Res(\phi))(u)\in A(G)^{\times,\perp}$, then $u\in A(H)^{\times,\perp}$.
\end{lemma}
\begin{proof}
Since $A(Res(\phi))$ is a $k$-algebra homomorphism, it can be restricted to a morphism between the unit groups, and it preserves orthogonal units since $\invol$ commutes with bisets. If $\phi$ is an epimorphism, then $Ind(\phi)\circ Res(\phi)\cong H$ and so $A(Res(\phi))$ is injective. In this case, if $A(Res(\phi))(u)\in A(G)^{\times,\perp}$ for a unit $u$, then $A(Res(\phi))(u^{\star})=A(Res(\phi))(u^{-1})$, hence $u^{\star}=u^{-1}$.
\end{proof}

\begin{prop}
Let $A$ be a $\star$-Green biset functor and $G\in Ob(\mathcal{D})$. The map $\widetilde{\_}:A(G)\longrightarrow End_{\AsCatA}(G)$ can be restricted to a group monomorphism $A(G)^{\times,\perp} \longrightarrow Aut_{\AsCatA}(G)^{\perp}$. Moreover, if $u \in A(G)^{\times}$ is such that $\widetilde{u}\in Aut_{\AsCatA}(G)^{\perp}$, then $u\in A(G)^{\times,\perp}$.
\end{prop}
\begin{proof}
The map $\widetilde{\_}$ is an injective $k$-algebra homomorphism, therefore it can be restricted to a monomorphism between the unit groups. By \thref{Dmorphd}, $\widetilde{\_}$ sends orthogonal units to orthogonal automorphisms, and if $\widetilde{u}$ is orthogonal then so is $u$.
\end{proof}

For every $N\unlhd G$, there is a monomorphism from $B(G/N,G/N)^{\times}$ to $B(G,G)^{\times}$ preserving orthogonal units \cite[Prop. 3.3]{Bar}. We now generalize this to an arbitrary Green biset functor.

\begin{lemma}\thlabel{EpiMul}
Let $A$ be a Green biset functor and $\phi:G\longrightarrow H$ be a group epimorphism with $G,H\in Ob(\mathcal{D})$. The map $\E_A(Res(\phi)):End_{\AsCatA}(H)\longrightarrow End_{\AsCatA}(G)$ is $k$-linear, multiplicative and injective. In particular, if $\phi$ is an isomorphism, then $\E_A(Iso(\phi))$ is a $k$-algebra isomorphism.
\end{lemma}
\begin{proof}
As in \cite[Lemma 3.1]{Bar}, this is a consequence of the fact that $Ind(\phi) \circ Res(\phi)\cong H$ and the functoriality of $\E_A$.
\end{proof}

\begin{prop}
Let $A$ be Green biset functor, $G\in Ob(\mathcal{D})$ and $N\unlhd G$. The map
\begin{align}
    End_{\AsCatA}(G/N)\longrightarrow End_{\AsCatA}(G), \alpha\mapsto Id_G+\E_A(Inf_{G/N}^G)(\alpha-Id_{G/N})
\end{align}
can be restricted to a group monomorphism
$$dAInf_{G/N}^G:Aut_{\AsCatA}(G/N)\longrightarrow Aut_{\AsCatA}(G).$$
Moreover, we have:
\begin{enumerate}
    \item $dAInf_{G/1}^G=\E_A(Inf_{G/1}^G)|_{Aut_{\AsCatA}(G/1)}$ is an isomorphism,
    \item If $M,N\unlhd G$ and $N\leq M$, then
    $$dAInf_{G/N}^G\circ dAInf_{G/M}^{G/N}=dAInf_{G/M}^G$$
    where we identify $G/M$ with the quotient $(G/N)/(M/N)$.
\end{enumerate}
\end{prop}
\begin{proof}
This follows from \cite[Prop. 3.2]{Bar} since $\E_A(Inf_{G/N}^G)$ is $k$-linear and multiplicative, and $dAInf_{G/N}^G$ is injective because $\E_A(Inf_{G/N}^G)$ is. Part (1) now follows from \thref{EpiMul} since $Inf_{G/1}^{G}\cong Iso(G/1\longrightarrow G)$, while Part (2) is a consecuence of the transitivity of inflations.
\end{proof}

\begin{prop}
Let $A$ be a $\star$-Green biset functor, $G\in Ob(\mathcal{D})$ and $N\unlhd G$. Then $dAInf_{G/N}^G(\omega^{\bullet})=dAInf_{G/N}^G(\omega)^{\bullet}$ for all $\omega \in Aut_{\AsCatA}(G/N)$, and $dAInf_{G/N}^G$ restricts to a group monomorphism $Aut_{\AsCatA}(G/N)^{\perp}\longrightarrow Aut_{\AsCatA}(G)^{\perp}$. Moreover, if $\omega \in Aut_{\AsCatA}(G/N)$ is such that $dAInf_{G/N}^G(\omega)\in Aut_{\AsCatA}(G)^{\perp}$, then $\omega \in Aut_{\AsCatA}(G/N)^{\perp}$.
\end{prop}
\begin{proof}
We only prove that $dAInf_{G/N}^G$ commutes with the duality, the rest will follow as in \thref{involnmorph}. If $\omega \in Aut_{\AsCatA}(G/N)$, then 
\begin{align*}
    dAInf_{G/N}^G(\omega)^{\bullet} &=\left(Id_G+e_{G\times G/N}\left(\overrightarrow{Inf_{G/N}^G}\right)\circ \left(\omega - Id_{G/N}\right)\circ e_{G/N\times G}\left(\overrightarrow{Def_{G/N}^G}\right)\right)^{\bullet}\\
    &=Id_G^{\bullet}+\left(e_{G/N\times G}\left(\overrightarrow{Def_{G/N}^G}\right)\right)^{\bullet}\circ \left(\omega - Id_{G/N}\right)^{\bullet}\circ \left(e_{G\times G/N}\left(\overrightarrow{Inf_{G/N}^G}\right)\right)^{\bullet}\\
    &=Id_G+e_{G\times G/N}\left(\overrightarrow{(Def_{G/N}^G)^{op}}\right)\circ \left(\omega^{\bullet} - Id_{G/N}^{\bullet}\right)\circ e_{G/N\times G}\left(\overrightarrow{(Inf_{G/N}^G)^{op}}\right)\\
    &=Id_G+e_{G\times G/N}\left(\overrightarrow{Inf_{G/N}^G}\right)\circ \left(\omega^{\bullet} - Id_{G/N}\right)\circ e_{G/N\times G}\left(\overrightarrow{Def_{G/N}^G}\right)\\
    &=dAInf_{G/N}^G(\omega^{\bullet}).
\end{align*}
\end{proof}

If $f:A\longrightarrow C$ is a morphism of $\star$-Green biset functors, then $f_G$ can be restricted to a group homomorphism $A(G)^{\times,\perp}\longrightarrow C(G)^{\times,\perp}$ for every $G$. 

\begin{prop}
Let $f:A\longrightarrow C$ be a morphism of $\star$-Green biset functors. Then $\mathcal{P}_f(\alpha^{\bullet})=\mathcal{P}_f(\alpha)^{\bullet}$ for all $\alpha \in \AsCatA(G,H)$ and all $G,H\in Ob(\mathcal{D})$, and $\mathcal{P}_f:Aut_{\AsCatA}(G)\longrightarrow Aut_{\AsCatC}(G)$ can be restricted to a group homomorphism $Aut_{\AsCatA}(G)^{\perp}\longrightarrow Aut_{\AsCatC}(G)^{\perp}$. Moreover, the diagrams
\begin{equation}
    \xymatrix{
    A(G)^{\times,\perp}\ar[rr]^{f_G}\ar[d]_{\widetilde{\_}} && C(G)^{\times,\perp}\ar[d]^{\widetilde{\_}}\\
    Aut_{\AsCatA}(G)^{\perp}\ar[rr]_{\mathcal{P}_f} && Aut_{\AsCatC}(G)^{\perp},
    }
\end{equation}
\begin{equation}
    \xymatrix{
    Aut_{\AsCatA}(G/N)^{\perp}\ar[rr]^{\mathcal{P}_f}\ar[d]_{dAInf_{G/N}^G} && Aut_{\AsCatC}(G/N)^{\perp}\ar[d]^{dCInf_{G/N}^G}\\
    Aut_{\AsCatA}(G)^{\perp}\ar[rr]_{\mathcal{P}_f} && Aut_{\AsCatC}(G)^{\perp}
    }
\end{equation}
commute for all $G\in Ob(\mathcal{D})$ and all $N\unlhd G$.
\end{prop}
\begin{proof}
Again, it is enough to prove that $\mathcal{P}_f$ commutes with the duality. Let $\alpha \in \AsCatA(G,H)$, then 
$$\mathcal{P}_f(\alpha^{\bullet}) =f_{G\times H}(A(Iso(\tau_{H,G}))(\alpha^{\star}))=C(Iso(\tau_{H,G}))(f_{H\times G}(\alpha)^{\star})=\mathcal{P}_f(\alpha)^{\bullet}.$$
\end{proof}


\end{document}